\documentclass[11pt,twoside]{article}
\usepackage{amsmath}
\usepackage{amsthm}
\makeatletter
\renewenvironment{proof}[1][\proofname]{%
	\par\pushQED{\qed}\normalfont%
	\topsep6\p@\@plus6\p@\relax
	\trivlist\item[\hskip\labelsep\bfseries#1\@addpunct{.}]%
	\ignorespaces
}{%
	\popQED\endtrivlist\@endpefalse
}
\makeatother 
\usepackage{mathtools}
\usepackage[mathscr]{eucal}
\usepackage{amssymb}
\usepackage{graphicx}
\usepackage{enumerate}
\usepackage[authoryear,sort&compress,round,comma]{natbib}
\newtheorem{definition}{Definition}
\usepackage{mathptmx}

\newtheorem{corollary}{Corollary}[section]
\newtheorem{theorem}{Theorem}[section]
\newtheorem{lemma}{Lemma}[section]
\newtheorem{remark}{Remark}[section]
\newtheorem{proposition}{Proposition}[section]
\newtheorem{example}{Example}[section]
\numberwithin{equation}{section}

\usepackage{graphicx,type1cm,eso-pic,color}

\makeatletter
\AddToShipoutPicture{
	\setlength{\@tempdimb}{.5\paperwidth}
	\setlength{\@tempdimc}{.5\paperheight}
	\setlength{\unitlength}{1pt}
	\put(\strip@pt\@tempdimb,\strip@pt\@tempdimc)
}
\makeatother

\begin{document}
	\setcounter{page}{1}

	\thispagestyle{empty}
	\markboth{}{}

	\pagestyle{myheadings}
	\markboth{}{ }
	
	\date{}
	
	
	\noindent  
	
	\vspace{.1in}
	
	{\baselineskip 20truept
		
		\begin{center}
			{\Large {\bf On General Weighted Extropy of Ranked Set Sampling }} \footnote{\noindent
				{\bf Nitin Gupta} E-mail: nitin.gupta@maths.iitkgp.ac.in\\
				{\bf Santosh Kumar Chaudhary } E-mail: skchaudhary1994@kgpian.iitkgp.ac.in}\\
			
		\end{center}
		
		\vspace{.1in}
		
		\begin{center}
			{\large {\bf Nitin Gupta and  Santosh Kumar Chaudhary}}\\
			{\large {\it Department of Mathematics, Indian Institute of Technology Kharagpur, West Bengal 721302, India }}
			\\
		\end{center}
	}
	\vspace{.1in}
	\baselineskip 12truept

	
	\begin{center}
		{\bf \large Abstract}\\
	\end{center}
	In the past six years, a considerable attention has been given to the extropy measure proposed by Lad et al. (2015). Weighted Extropy of Ranked Set Sampling was studied and compared with simple random sampling by Qiu et al. (2022). The general weighted extropy and some results related to it are introduced in this paper. We provide general weighted extropy of ranked set sampling. We also studied characterization results, stochastic comparison and monotone properties of general weighted extropy.\\
	\\
	\textbf{Keyword:} Extropy, General Weighted Extropy, Ranked Set Sampling, Stochastic Order, Weighted Extropy.\\
	\newline
	\noindent  {\bf Mathematical Subject Classification}: {\it 62B10, 62D05}
	\section{Introduction}
	McIntyre (1952) introduced the ranked set sampling (RSS)  to estimate mean pasture yields. The RSS is a better sampling strategy for estimating population mean than simple random sampling (SRS).
	
	Let $X$ has the probability density function (pdf) $f$, the cumulative distribution function (cdf) $F$ and the survival function (sf) $\bar F=1-F$.  Assume that $\mathbf{X}_{SRS}=\{X_i:\ i=1,\ldots,n\}$ denote a simple random sample of size $n$ from $X$. Here we first explain the one-cycle RSS: randomly select $n^2$ units from $X$, then these units are randomly allocated into $n$ sets, each of size $n$. Then rank those $n$ units in each set with respect to the variable of interest. From the first set, select the smallest ranked unit; then from the second set select the second smallest ranked unit; continue the process until the $n$th smallest ranked unit is selected from the last set. Now repeat the entire sampling procedure $m$ times, to obtain a sample of size $mn$. The obtained sample is called the RSS from the underlying distribution $F$. We assume $m=1$ throughout the manuscript without loss of generality.  Let $\mathbf{X}_{RSS}^{(n)}=\{X_{(i:n)i},\ i=1,\ldots,n\}$ denote the  ranked-set sample, where $X_{(i:n)i}$ represent the $i$-th order statistics from the $i$-th sample with sample size $n$.
	
    Entropy was first introduced by Shannon (1948) and is relevant to many disciplines, including information theory, physics, probability and statistics, economics, communication theory, etc. The average level of uncertainty related to the results of the random experiment is measured by entropy.
	
	Entropy measures the average level of uncertainty related to the results of the random experiment. The differential form of Shannon entropy  is defined  as:
	\begin{equation*}
	\label{1eq2}
		H(X) = - {\int_{-\infty}^{\infty}{f(x) \ln \left( f(x)\right)dx}}=E\left(- \ln f(X)\right).
	\end{equation*} 
	This measure is used in many different applications in order statistics and record statistics; see, for instance, Bratpour et al. (2007), Raqab and Awad (2000, 2001), Zarezadeh and Asadi (2010), Abo-Eleneen (2011),  Qiu and Jia (2018a,2018b) and Tahmasebi et al. (2016).
	
	Recently an alternative measure of uncertainty, called extropy
	has gained importance. Lad et al. (2015) defined the
	compliment dual of the Shannon entropy called extropy as		
	\begin{equation}\label{extropy}
	J(X)=-\frac{1}{2} \int_{-\infty}^{\infty}f^2(x)dx=-\frac{1}{2}E\left(f(X)\right).
	\end{equation}

    Qiu (2017) studies the characterization results, lower bounds, monotone properties and statistical applications concerning extropy of order statistics and record values.  Balakrishnan et al. (2020) and Bansal and Gupta (2021) independently introduced the weighted extropy as
	\begin{equation}\label{wtdextropy}
	J^x(X)=-\frac{1}{2}\int_{-\infty}^{\infty}xf^2(x)dx
	\end{equation}
	
	Balakrishnan et al. (2020) studied characterization results and bounds 
	for weighted versions of extropy, residual extropy, past extropy, bivariate extropy
	and bivariate weighted extropy, whereas Bansal and Gupta (2021) discussed the
	results for weighted extropy and weighted residual extropy for order statistics and
	k-record values. Here we introduce the  generalized weighted extropy with weight $w_1(x)\geq 0$ as
	\begin{align*}
		J^{w_1}(X)&=-\frac{1}{2} \int_{-\infty}^{\infty}w_1(x) f^2(x)dx\\
		&=-\frac{1}{2}E(\Lambda_X^{w_1}(U)),
	\end{align*}
	where $\Lambda_X^{w_1}(u)=w_1(F^{-1}(u))f(F^{-1}(u))$ and $U$ is uniformaly distributed random variable on $(0,1)$, i.e., $U\sim$ Uniform$(0,1)$.
	
	Bansal and Gupta (2021) used following example in which extropy of random variable X and Y are same but weighted extropy are different. We find that general weighted extropies  $J^{w_1}(X)$ and $J^{w_1}(Y)$ are also different when $w_1(x)=x^m, \ m>0, \ x>0.$ Let X and Y be two random variables with the pdf's 
	\begin{eqnarray*}
		f_{X}(x)=
		\begin{cases}
			2x, \hspace{4mm} 0<x<1\\
			0, \hspace{5mm} otherwise
		\end{cases}~~~~~~~~~~~~~~~~~
		f_{Y}(x)=
		\begin{cases}
			2(1-x),\hspace{4mm} 0<x<1\\
			0,\hspace{6mm} otherwise
		\end{cases}
	\end{eqnarray*}
We get $J(X)=J(Y)=-2/3$ but  $J^x(X)=-1/2$ and $J^x(Y)=-1/6.$ Let us consider $w_1(x)=x^m,  \ m>0, \ x>0 $ then 
\begin{align*}
	J^{w_1}(X)&=-\frac{2}{m+3} \\
	J^{w_1}(Y)&=-2 \left( \frac{1}{m+1} - \frac{2}{m+2}+ \frac{1}{m+3}\right) 
\end{align*}
    
  As a result, while extropies in this example are same, weighted extropies and general weighted extropies are different. Therefore, weighted extropies and general weighted extropies can be used as uncertainty measures. This shift dependent measure takes into consideration the values of the random variable, unlike the extropy defined in (\ref{extropy}). 
 
    Qiu et al. (2022) have given a representation for the weighted extropy of ranked set sampling in terms of quantile and density-quantile functions. Then  provided some related results including monotone properties, stochastic orders, characterizations and sharp bounds. Moreover, They have shown how the weighted extropy of ranked set sampling compares with its counterpart of simple random sampling.
 
     In this paper, we also study the monotone and stochastic properties of  general weighted extropy of RSS data. Stochastic comparison results are obtained by taking different weights for the extropy of RSS.  Comparison between extropy of RSS  and SRS data is provided. Some characterization results are obtained.  We also study the monotone properties for the general weighted extropy of RSS data. These results generalize some of the results available in the literature.
	
	
	\section{Result on general weighted extropies}
	Before providing results, let us review definitions from literature (see, Shaked and Shanthikumar (2007) of some terminology that are useful.
	
	\begin{definition}
		Let a random variable $X$ have  pdf $f(x)$, cdf $F(x)$ and sf $\bar{F}(x) = 1-F(x).$ Let $l_X = inf\{x \in \mathbb{R} : F(x) > 0\},\   u_X = sup\{x \in \mathbb{R} : F(x) < 1\}$ and	 $S_X= (l_X,u_X),$ where -$\infty \le l_X \le u_X \le \infty$.\\
		\\
		(i) $X$ is said to be log-concave (log-convex) if $\{x \in \mathbb{R} : f(x) > 0\} = S_X$
		and $ln(f(x))$ is concave (convex) on $S _X$.\\
		\\
		(ii) $X$ is said to have increasing (decreasing) failure rate IFR (DFR) if F(x) is log-concave		(log-convex) on  $S_X$.\\
		\\
		(iii) $X$ is said to have decreasing (increasing) reverse failure rate DRFR (IRFR) if F(x) is	log-concave (log-convex) on $S_X$.\\
		\\
		(iv) $X$ is said to have decreasing (increasing) mean residual life DMRL (IMRL) if $\int_{x}^{u_X} F(t)dt$		is log-concave (log-convex) on $S_X$.\\
		\\
		(v) $X$ is said to have increasing (decreasing) mean inactivity time (IMIT (DMIT)) if $\int_{l_X}^{x}	F(t)dt$
		is log-concave (log-convex) on $S_X.$
		
	\end{definition}

		\begin{definition}
		Let $X$ be a random variable with  pdf $f(x)$, cdf $F(x)$ and sf $\bar{F}(x)=1-F(x).$ Let $l_X = inf\{x \in \mathbb{R} : F(x) > 0\}$, $u_X = sup\{x \in \mathbb{R} : F(x) < 1\}$
		and $S_X = (l_X, u_X).$ Similarly, let $Y$ be a random variable with  pdf $g(x)$, cdf $G(x)$ and sf $\bar{G}(x)= 1-G(x).$ Let $l_Y=inf\{x \in R:G(x)> 0\},\  u_Y = sup\{x \in R :G(x) < 1\}$ and $S_Y = (l_Y,u_Y )$. If $l_X \ge 0$ and $l_Y \ge 0$, then \\
		\\
		(i) $X$ is said to be smaller than $Y$ in usual stochastic (st) ordering		$(X \le_{st} Y )$ if $\bar{F}(x) \le \bar{G}(x)$, for every -$\infty< x <\infty.$\\
		\\
		(ii) $X$ is said to be smaller than $Y$ in the likelihood ratio (lr) ordering $(X \le_{lr} Y )$ if
		$g(x)f(y) \le  f(x)g(y)$, whenever $-\infty < x < y <\infty$.\\
		\\
		(iv) $X$ is said to be smaller than $Y$ in the dispersive ordering ($X \le_{disp} Y)$ if $G^{-1}F(x)-x$ is increasing in $x \ge 0$.\\
		\\
		(v) $X$ is said to be smaller than $Y$ in the hazard rate ordering $(X \le_{hr} Y)$ if $\frac{\bar{G}(x)}{\bar{F}(x)}$  is increasing in $x\in S_X \cap S_Y.$
		
	\end{definition}

	\begin{theorem}\label{thm com en1}
		Let $X$  and $Y$ be nonnegative random variables with pdf's f and g, cdf's F and G, respectively having $u_X=u_Y<\infty$.\\
		\\
		(a) If $w_1$ is increasing, $w_1(x)\geq w_2(x)$ and $X\le_{disp} Y$, then $J^{w_1}(X)\le J^{w_2}(Y)$.\\
		(b)   If $w_1$ is increasing, $w_1(x)\leq w_2(x)$ and $X\ge_{disp} Y$, then $J^{w_1}(X)\ge J^{w_2}(Y)$.
	\end{theorem}
	\begin{proof}
		(a) Since $X\le_{disp} Y$, therefore we have $f(F^{-1}(u))\ge g(G^{-1}(u))$ for $u \in (0,1)$. Then using Theorem 3.B.13(b) of Shaked and Shanthikumar (2007), $X\le_{disp} Y$ implies that $X\ge_{st} Y$. Hence $F^{-1}(u) \ge G^{-1}(u)$ for all $u\in (0,1)$. Since $w_1$ is increasing and $w_1(x)\geq w_2(x)$ , then $w_1(F^{-1}(u)) \ge w_1(G^{-1}(u))\ge w_2(G^{-1}(u))$. 
		Hence 
		\begin{align}\label{stor0}
			\Lambda_X^{w_1} (u)&=w_1(F^{-1}(u)) f(F^{-1}(u))\nonumber\\
			&\ge w_2(G^{-1}(u)) g(G^{-1}(u))\nonumber\\
			=&\Lambda_Y^{w_2} (u).
		\end{align}
		Now
		\begin{align*}
			J^{w_1}(X)&= -\frac{1}{2}E\left(\Lambda_X^{w_1} (U)\right)\\
			&\le -\frac{1}{2}E\left(\Lambda_Y^{w_2} (U)\right)\\
			&=J^{w_2}(Y),
		\end{align*}
		the inequality here follows using (\ref{stor0}). Hence the result. \\
		(b) On similar arguments as in part (a), result follows. \hfill $\blacksquare$
	\end{proof}\
	If we take $w_1(x)=w_2(x)$ in above theorem, then we have following corollary.
	
	\begin{corollary}\label{cor com en1}
		Let $X$  and $Y$ be nonnegative random variables with pdf's f and g, cdf's F and G, respectively having $u_X=u_Y<\infty$. Let $w_1$ is increasing. \\
		(a) If  $X\le_{disp} Y$, then $J^{w_1}(X)\le J^{w_1}(Y)$.\\
		(b) If  $X\ge_{disp} Y$, then $J^{w_1}(X)\ge J^{w_1}(Y)$.
	\end{corollary}
 \noindent	The pdf of $X_{(i:n)i}$   is 
	\begin{align*}
		f_{i:n}(x)=\frac{n!}{(i-1)! (n-i)!} F^{i-1}(x) \bar{F}^{n-i}(x)f(x) ,\hspace{3mm} -\infty<x<\infty.
	\end{align*}
	Let
	\begin{align*}
		\phi_{2i-1:2n-1}(u)= \frac{(2n-1)!}{(2i-2)! (2n-2i)!} u^{2i-2} (1-u)^{2n-2i},\hspace{3mm} 0<u<1.
	\end{align*}
    Then $\phi_{2i-1:2n-1}(u)$ represents the pdf of a beta distributed random variable with parameter (2i-1), and (2n-2i+1), we denote this random variable by $B_{2i-1:2n-1}$.
	\begin{theorem}
		Consider a random sample of size $n$ as $X_1,\cdots,X_n$  from a IRFR distribution $F$.\\
		(a) If $w_1$ is increasing, then $J^{w_1}(X_{i:n})$ is decreasing in $i$ for fixed $n,\ 1\le i \le n$.\\
		(b) If $w_1$ is decreasing, then $J^{w_1}(X_{i:n})$ is increasing in $n$ for fixed $i,\ 1\le i \le n$.
	\end{theorem}
	\begin{proof} (a) The weighted extropy of $X_{i:n}$ is
		\begin{align*}
			J^{w_1}(X_{i:n})&=-\frac{n c_{i,n}}{2}\int_{0}^{1}\Lambda_{X}^{w_1}(u)\phi_{2i-1:2n-1}(u)du\\
			&=-\frac{c_{i,n}(2i-1)}{4}\int_{0}^{1}w_1(F^{-1}(u))r\left(F^{-1}(u)\right)\phi_{2i:2n}(u)du\\
			&=-\frac{c_{i,n}(2i-1)}{4}E\left(M(B_{2i:2n})\right),
		\end{align*}
		where 	$c_{i,n}=\frac{\binom{2i-2}{i-1}\binom{2n-2i}{n-i}}{\binom{2n-1}{n-1}}$, $M(u)=w_1(F^{-1}(u))r\left(F^{-1}(u)\right)$ and $r(x)=f(x)/F(x)$ denote the reverse failure rate function of $X$. Now, it follows that
		\begin{align*}
			\frac{J^{w_1}(X_{i:n})}{J^{w_1}(X_{i+1:n})}&=\frac{c_{i,n}}{c_{i+1,n}}\left(\frac{2i-1}{2i+1}\right)
			\frac{E\left(M(B_{2i:2n})\right)}{E\left(M(B_{2i+2:2n})\right)}\\
			&=\frac{i(2n-2i-1)}{(2i+1)(n-i)}\frac{E\left(M(B_{2i:2n})\right)}{E\left(M(B_{2i+2:2n})\right)}\\
			&\le \frac{E\left(M(B_{2i:2n})\right)}{E\left(M(B_{2i+2:2n})\right)}.
		\end{align*}
		Since $B_{2i:2n}\le_{hr}B_{2i+2:2n}$, hence $B_{2i:2n}\le_{st}B_{2i+2:2n}$, hence under the hypothesis of the 
		theorem it follows that 
		\begin{align*}
			E\left(M(B_{2i:2n})\right)&\le E\left(M(B_{2i+2:2n})\right),
		\end{align*}
		which further implies that  $J^{w_1}(X_{i:n})\geq J^{w_1}(X_{i+1:n})$. This completes the proof of part (a).\\
	
		(b) Proceeding in same fashion as in part (a), under the assumption of part (b) result follows by observing that
		\begin{align*}
			\frac{J^{w_1}(X_{i:n})}{J^{w_1}(X_{i:n+1})}&=\frac{(2n+1)(n-i+1)}{(n+1)(2n-2i+1)}\left(\frac{2i-1}{2i+1}\right)
			\frac{E\left(M(B_{2i:2n})\right)}{E\left(M(B_{2i:2n+2})\right)}\\
			&\ge \frac{E\left(M(B_{2i:2n})\right)}{E\left(M(B_{2i:2n+2})\right)}
		\end{align*}
		and $B_{2i:2n}\ge_{st}B_{2i:2n+2}$. \hfill $\blacksquare$
	 
	\end{proof}

	\section{\textbf{General Weighted Extropies of RSS}}	
	Let $X$ be a random variable  with finite mean $\mu$ and variance $\sigma^2$. For $\textbf{X}_{SRS}=\{X_i,\ i=1,\ldots,n\}$, the joint pdf is $\prod_{i=1}^{n}f(x_i)$, as $X_i$'s, $i=1,\ldots,n$ are independent and identically distributed (i.i.d.). Hence the extropy of $\textbf{X}_{SRS}^{(n)}$ can be defined as
	\begin{align}\label{SRS1}
		J^{w_1}(\textbf{X}_{SRS}^{(n)})&=\frac{-1}{2}\prod_{i=1}^{n}\left(\int_{-\infty}^{\infty}w_1(x_i)f^2(x_i)dx_i\right)\nonumber \\
		&=\frac{-1}{2}\left(-2J^{w_1}(X)\right)^n\nonumber \\
		&=\frac{-1}{2}\left(E(\Lambda_X^{w_1}(U))\right)^n
	\end{align}
	 Now, we can write the general weighted extropy of   $\textbf{X}_{RSS}^{(n)}$  as
	\begin{align}\label{RSS1} J^{w_1}(\textbf{X}_{RSS}^{(n)})&=-\frac{1}{2}\prod_{i=1}^{n}\left(-2J^{w_1}(X_{(i:n)i})\right)\nonumber \\
		&=-\frac{1}{2}\prod_{i=1}^{n} \int_{-\infty}^{\infty}n^2 \binom{n-1}{i-1}^2w_1(x) F^{2i-2}(x)\bar F^{2n-2i}(x)f^2(x)dx \nonumber\\
		&=-\frac{1}{2}\prod_{i=1}^{n} \int_{0}^{1}n^2 \binom{n-1}{i-1}^2\Lambda_X^{w_1}(u) u^{2i-2} (1-u)^{2n-2i}du \nonumber\\
		&=- \frac{Q_n}{2}\prod_{i=1}^{n}E\left(\Lambda_X^{w_1} (B_{2i-1:2n-1})\right),
	\end{align}
	where \begin{align*}Q_n&=n^n\prod_{i=1}^{n}c_{i,n},\\
		c_{i,n}&=\frac{\binom{2i-2}{i-1}\binom{2n-2i}{n-i}}{\binom{2n-1}{n-1}}\end{align*}
	and $B_{2i-1:2n-1}$ is a beta distributed random variable with parameters $(2i-1)$ and $(2n-2i+1)$. Equation $(\ref{RSS1})$ provides an expression in simplified form of the general weighted extropy	of $\textbf{X}_{RSS}^{(n)}$. Now we provide some examples to illustrate the equation $(\ref{RSS1})$ .

	

	\begin{example}
		Let $W$ be a random variable with power distribution. The pdf and cdf of W are respectively $f(x)= \theta x^{\theta -1}$ 
		and $F(x)= x^{\theta}$ , $0<x<1$ ,  $\theta > 0$. Let $w_1(x)=x^m,\ x>0, \ m>0$, then
		it follows that 
		\[\Lambda_W^{w_1}(u)= w_1(F^{-1}(u)) f(F^{-1}(u))=\theta u^{\frac{m + \theta - 1}{\theta}}\]
		for $ w_1(x)=x^m.$ 
		Then we have
		\begin{align*}
			& J^{w_1}(\textbf{W}_{RSS}^{(n)}) \\
			&= -\frac{Q_n}{2}\prod_{i=1}^{n}E\left(\Lambda_W^{w_1} (B_{2i-1:2n-1})\right)\\
			&= - \frac{Q_n}{2}\prod_{i=1}^{n}\int_{0}^{1} \Lambda_W^{w_1}(u) \frac{(2n-1)!}{(2i-2)!(2n-2i)!}u^{2i-2} (1-u)^{2n-2i}du \\
			&= - \frac{Q_n}{2}\prod_{i=1}^{n}\int_{0}^{1} \theta u^{\frac{m + \theta - 1}{\theta}} \frac{(2n-1)!}{(2i-2)!(2n-2i)!}u^{2i-2} (1-u)^{2n-2i}du \\
			&=- \frac{Q_n}{2} {\theta}^n \prod_{i=1}^{n}\int_{0}^{1}  u^{\frac{m + \theta - 1}{\theta}} \frac{(2n-1)!}{(2i-2)!(2n-2i)!}u^{2i-2} (1-u)^{2n-2i}du\\
			&=- \frac{Q_n}{2} {\theta}^n \prod_{i=1}^{n} \frac{(2n-1)!  \Gamma(\frac{m+2i\theta-1}{\theta})}{(2i-2)!  \Gamma(m+(2n+1)\theta-1)}.
		\end{align*}\hfill $\blacksquare$
	\end{example}
	
	\begin{example}
		Let $Z$ have an exponential distribution with cdf $F_Z(z)=1-e^{-\lambda z}, \ \lambda >0, \ z>0$. Let $w_1(x)=x^m, \ m>0, \ x>0$, then it follows that
		\begin{align*}
			\Lambda_Z^{w_1} (u) = w_1(F^{-1}(u)) f(F^{-1}(u))
			= \frac{(-1)^m (1-u) (ln(1-u))^m}{\lambda^{m-1}}, 0<u<1.  	
		\end{align*}
		Then we have
		\begin{align*}
			\hspace{20pt} J^{w_1}(\textbf{Z}_{RSS}^{(n)}) 
			&= -\frac{Q_n}{2}\prod_{i=1}^{n}E\left(\Lambda_Z^{w_1} (B_{2i-1:2n-1})\right)\\
			&= - \frac{Q_n}{2}\prod_{i=1}^{n}\int_{0}^{1} \Lambda_Z^{w_1}(u) \frac{(2n-1)!}{(2i-2)!(2n-2i)!}u^{2i-2} (1-u)^{2n-2i}du \\
			&= - \frac{Q_n}{2}\prod_{i=1}^{n}\int_{0}^{1} \frac{(-1)^m (1-u) (ln(1-u))^m}{\lambda^{m-1}} \frac{(2n-1)!}{(2i-2)!(2n-2i)!}\\
			& \ \ \ \ \ \ \ \ \ \ \ \ \ \ \ \ \ \ \ \ \ \ \ \ \ \ \ \ \ \ \ \ \ \ \ \ \ \ \ \  \ \ \ \ \ \ \ \ \ \ \ \ \ \ \ \ \ \ \ \ \ \ \ \ \ \ \ \ \ \ \     \ \ \ \ \     u^{2i-2}(1-u)^{2n-2i}du\\
			&=- \frac{Q_n}{2}\prod_{i=1}^{n} \frac{(-1)^m}{\lambda^{m-1}}\int_{0}^{1} (ln(1-u))^m \frac{(2n-1)!}{(2i-2)!(2n-2i)!}\\
			& \ \ \ \ \ \ \ \ \ \ \ \ \ \ \ \ \ \ \ \ \ \ \ \ \ \ \ \ \ \ \ \ \ \ \ \ \ \ \ \  \ \ \ \ \ \ \ \ \ \ \ \ \ \ \ \ \ \ \ \ \ \ \ \ \ \ \ \ \ \ \     \ \    u^{2i-2}(1-u)^{2n-2i+1}du.
		\end{align*}
		Taking $u=1-e^{-x}$ in the above equation, we get
		\begin{align*}
			& J^{w_1}(\textbf{Z}_{RSS}^{(n)})\\ 
			&=- \frac{Q_n}{2}\prod_{i=1}^{n} \frac{(-1)^m}{\lambda^{m-1}}\int_{0}^{\infty} (ln(e^{-x}))^m \frac{(2n-1)!}{(2i-2)!(2n-2i)!}\\
			& \ \ \ \ \ \ \ \ \ \ \ \ \ \ \ \ \ \ \ \ \ \ \ \ \ \ \ \ \ \ \ \ \ \ \ \ \ \ \ \  \ \ \ \ \ \ \ \ \ \ \ \ \ \ \ \ \ \ \ \    (1-e^{-x})^{2i-2}(e^{-x})^{2n-2i+1} e^{-x}dx\\
			&=- \frac{Q_n}{2}\prod_{i=1}^{n} \frac{1}{\lambda^{m-1}}\int_{0}^{\infty} x^m \frac{(2n-1)!}{(2i-2)!(2n-2i)!}(1-e^{-x})^{2i-2}(e^{-x})^{2n-2i+2}dx\\
			&=- \frac{Q_n}{2} \frac{1}{\lambda^{n(m-1)}} \prod_{i=1}^{n} \int_{0}^{\infty} x^m \frac{(2n-1)!}{(2i-2)!(2n-2i)!}(1-e^{-x})^{2i-2}(e^{-x})^{2n-2i+2}dx\\
			&=- \frac{Q_n}{2} \frac{1}{\lambda^{n(m-1)}} \prod_{i=1}^{n} \left[ \left( \frac{2n-2i+1}{2n}\right) \int_{0}^{\infty} x^m \frac{(2n)!}{(2i-2)!(2n-2i+1)!}\right.\\
			& \ \ \ \ \ \ \ \ \ \ \ \ \ \ \ \ \ \ \ \ \ \ \ \ \ \ \ \ \ \ \ \ \ \ \ \ \ \ \ \  \ \ \    \ \ \ \ \ \left. (1-e^{-x})^{2i-2}(e^{-x})^{2n-2i+2}dx \right]\\
			&=- \frac{Q_n (2n-1)!!}{2^{n+1} n^n} \frac{1}{\lambda^{n(m-1)}} \prod_{i=1}^{n}  E(Z_{2i-1:2n}^{m}), \\
		\end{align*}
		where $Z_{2i-1:2n}$ is the $(2i-1)$-th order statistics of a sample of size $2n$ from exponential distribution having pdf given by 
		$$\phi_{2i-1:2n}=\frac{(2n)!}{(2i-2)!(2n-2i+1)!}(1-e^{-x})^{2i-2}(e^{-x})^{2n-2i+2},\  x\ge 0,$$ 
		and  $(2n-1)!!=\prod_{i=1}^{n} (2n-2i+1)=\prod_{i=1}^{n} (2i-1)$.\hfill $\blacksquare$
	\end{example}
	
	\begin{example}
		Let $V$ be a Pareto random variable with cdf $F(x)=1-x^{-\alpha},\  \alpha >0, \ x>1 $. Let $w_1(x)=x^m,\ m>0, \ x>0 $, then we get
		\begin{align*}
			\Lambda_V^{w_1} (u) 
			&= w_1(F^{-1}(u)) f(F^{-1}(u))\\
			&=\alpha (1-u)^{\frac{\alpha - m+1}{\alpha}}.		
		\end{align*}
		The weighted extropy of $\textbf{V}_{RSS}^{(n)}$ is 
		\begin{align*}
			J^{w_1}(\textbf{V}_{RSS}^{(n)}) 
			&= -\frac{Q_n}{2}\prod_{i=1}^{n}E\left(\Lambda_V^{w_1} (B_{2i-1:2n-1})\right)\\
			&= - \frac{Q_n}{2}\prod_{i=1}^{n}\int_{0}^{1} \Lambda_V^{w_1}(u) \frac{(2n-1)!}{(2i-2)!(2n-2i)!}u^{2i-2} (1-u)^{2n-2i}du \\
			&= - \frac{Q_n}{2}\prod_{i=1}^{n}\int_{0}^{1} \alpha (1-u)^{\frac{\alpha - m+1}{\alpha}} \frac{(2n-1)!}{(2i-2)!(2n-2i)!}u^{2i-2} (1-u)^{2n-2i}du \\
			&= - \frac{Q_n}{2} \alpha^n \prod_{i=1}^{n}\int_{0}^{1}  \frac{(2n-1)!}{(2i-2)!(2n-2i)!}u^{2i-2} (1-u)^{2n-2i+ \frac{\alpha - m+1}{\alpha}}du \\
			&=  - \frac{Q_n}{2} \alpha^n \prod_{i=1}^{n} \left[ \frac{(2n-1)!}{(2n-2i)!} \frac{\Gamma(\frac{2n\alpha-2i\alpha+2\alpha-m+1}{\alpha})}{\Gamma(\frac{\alpha+2n\alpha-m+1}{\alpha})} \right].
		\end{align*}\hfill $\blacksquare$

	\end{example}


		
		The following result gives the conditions under which the general weighted extropy will increase (decrease).	
		
		\begin{theorem}
			Let $X$ is a non-negative absolutely continuous random variable with pdf f and cdf F. Assume $\phi(x)$ is a increasing function and  $\frac{w_1(\phi(x))}{\phi^\prime (x)} \leq (\geq) w_1(x)$ and $\phi(0)=0$. If $Z=\phi(X)$, then $J^{w_1}(\textbf{X}_{RSS}^{(n)})\leq (\geq) J^{w_1}(\textbf{Z}_{RSS}^{(n)})$

		\end{theorem}
		\begin{proof}
			Let pdf and cdf of $Z$ be $h$ and $H$, respectively. Then 
			\begin{align*}
				\Lambda_Y^{w_1}(u)&=w_1\left(H^{-1}(u)\right) h\left(H^{-1}(u)\right)\\
				&=\frac{w_1(\phi(F^{-1}(u)))}{\phi^\prime(F^{-1}(u))}f(F^{-1}(u)) \hspace{5mm} \forall \hspace{5mm} 0<u<1.
			\end{align*}
			Note that $\phi(x)\geq \phi(0)$, $\forall$ $x\geq 0$. Hence for $ 0<u<1$, we have
			\begin{align*}
				\Lambda_Y^{w_1}(u)=\frac{w_1(\phi(F^{-1}(u)))}{\phi^\prime(F^{-1}(u))}f(F^{-1}(u))\leq w_1(F^{-1}(u))f(F^{-1}(u))=\Lambda_X^{w_1}(u).
			\end{align*}
			Therefore $J^{w_1}(\textbf{X}_{RSS}^{(n)})\leq J^{w_1}(\textbf{Y}_{RSS}^{(n)})$ using equation (\ref{RSS1})\hfill. Proof of other part can be done in similar fashion.\hfill $\blacksquare$
		    \end{proof}
		
		\noindent We now give a lower bound for the general weighted extropy of RSS data. This lower bound is dependent on the weighted extropy of the SRS data, as shown by the following result.
		
		\begin{theorem}
			Let $X$ be an absoultely continouous random variable with pdf f and cdf F. Then for $n\geq 2$,
		\end{theorem}
		\begin{align*}
			\frac{ J^{w_1}(\textbf{X}_{RSS}^{(n)})}{ J^{w_1}(\textbf{X}_{SRS}^{(n)})}\leq \frac{n^{2n}}{(n-1)^{2(n-1)(n-2)}}\prod_{i=2}^{n-1}\left(\binom{n-1}{i-1}^2(i-1)^{2i-2}(n-i)^{2n-2i}\right).
		\end{align*}
		\begin{proof}
			Proof is on similar lines as proof of  theorem 2.8 of Qiu and Raqab (2022).\hfill $\blacksquare$
		\end{proof}

		\section {Characterization results}

		\begin{theorem}
			Let $X$ be an absolutely continuous random variable with pdf f and cdf F; and assume $w_1(-x)=-w_1(x)$. Then $X$ is symmetric distributed  random variable with mean 0 if and only if	$J^{w_1}(\textbf{X}_{RSS}^{(n)})=0$  for all odd $n\geq 1$. 
		\end{theorem}
		\begin{proof}
			For sufficiency, suppose $f(x)=f(-x)$ for all $x\geq 0$. Also since $F^{-1}(u)=- F^{-1}(1-u)$,  $f(F^{-1}(u))=f(F^{-1}(1-u))$ for all $0<u<1$ and $w_1(-x)=-w_1(x)$, which implies that
			\begin{equation*}
				\Lambda_X^{w_1} (u)=w_1(F^{-1}(u))f(F^{-1}(u))=-w_1(F^{-1}(1-u)) f(F^{-1}(1-u))=-\Lambda_X^{w_1} (1-u)
			\end{equation*}
			In the similar fashion as in Qiu and Raqab (2022),   $J^{w_1}(\textbf{X}_{RSS}^{(n)})=-J^{w_1}(\textbf{X}_{RSS}^{(n)})$ .
			This completes the proof of sufficiency.\\
			For the necessity, since equation $J^{w_1}(\textbf{X}_{RSS}^{(n)})=0$ holds for all odd $n\geq 1$. For $n=1$, \[J^{w_1}(\textbf{X}_{RSS}^{(1)})=J^w(X)=0.\]
			Now,
			\begin{align*}
				J^{w_1}(X)&=-\frac{1}{2}\int_{-\infty}^{\infty}w_1(x)f^2(x)dx\\
				&=-\frac{1}{2}\left(\int_{-\infty}^{0}w_1(x)f^2(x)dx+\int_{0}^{\infty}w_1(x)f^2(x)dx\right)\\
				&=-\frac{1}{2}\left(\int_{-\infty}^{0}w_1(x)f^2(x)dx+\int_{0}^{\infty}w_1(x)f^2(x)dx\right)\\
				&=-\frac{1}{2}\left(-\int_{0}^{\infty}w_1(x)f^2(-x)dx+\int_{0}^{\infty}w_1(x)f^2(x)dx\right)\\
				&=-\frac{1}{2}\left(-\int_{0}^{\infty}w_1(x) \left(f(x)+f(-x)\right)\left(f(x)-f(-x)\right)dx\right)\\
				&=0,
			\end{align*}
			since $w_1(x)>0$, and $f(x)=f(-x)$ $\forall$ $x\geq 0$. This provides the proof of necessity. \hfill $\blacksquare$
		\end{proof}
		
		\section{Stochastic comparision}
		In the following result, we provide the conditions for comparing two RSS schemes under different weights.
		
		\begin{theorem}\label{thm com rss1}
			Let $X$  and $Y$ be nonnegative random variables with pdf's f and g, cdf's F and G, respectively having $u_X=u_Y<\infty$.\\
			(a) If $w_1$ is increasing, $w_1(x)\geq w_2(x)$ and $X\le_{disp} Y$, then $J^{w_1}(\textbf{X}_{RSS}^{(n)})\le J^{w_2}(\textbf{Y}_{RSS}^{(n)})$.\\
			(b)   If $w_1$ is increasing, $w_1(x)\leq w_2(x)$ and $X\ge_{disp} Y$, then $J^{w_1}(\textbf{X}_{RSS}^{(n)})\ge J^{w_2}(\textbf{Y}_{RSS}^{(n)})$.
		\end{theorem}
		\begin{proof}
			(a) Since $X\le_{disp} Y$, therefore we have $f(F^{-1}(u))\ge g(G^{-1}(u))$ for all $u\in (0,1)$. Then using Theorem 3.B.13(b) of Shaked and Shanthikumar (2007), $X\le_{disp} Y$ implies that $X\ge_{st} Y$. Hence $F^{-1}(u) \ge G^{-1}(u)$ $\forall$ $u\in (0,1)$. Since $w_1$ is increasing and $w_1(x)\geq w_2(x)$, then $w_1(F^{-1}(u)) \ge w_1(G^{-1}(u))\ge w_2(G^{-1}(u))$. 
			Hence 
			\begin{align}\label{stor1}
				\Lambda_X^{w_1} (u)&=w_1(F^{-1}(u)) f(F^{-1}(u))\nonumber\\
				&\ge w_2(G^{-1}(u)) g(G^{-1}(u))\nonumber\\
				=&\Lambda_Y^{w_2} (u).
			\end{align}
			Now,  
			\begin{align*}
				J^{w_1}(\textbf{X}_{RSS}^{(n)})&=- \frac{Q_n}{2}\prod_{i=1}^{n}E\left(\Lambda_X^{w_1} (B_{2i-1:2n-1})\right)\\
				&\le - \frac{Q_n}{2}\prod_{i=1}^{n}E\left(\Lambda_Y^{w_2} (B_{2i-1:2n-1})\right)\\
				&=J^{w_2}(\textbf{Y}_{RSS}^{(n)}),
			\end{align*}
			the inequality here follows using (\ref{stor1}). Hence the result. \\
			(b) On similar arguments as in part (a), result follows.\hfill $\blacksquare$
		\end{proof}
		
		If we take $w_1(x)=w_2(x)$ in above theorem, then The following corollary follows.
		
		\begin{corollary}\label{cor com rss1}
			Let $X$  and $Y$ be nonnegative random variables with pdf's f and g, cdf's F and G, respectively having $u_X=u_Y<\infty$; let $w_1$ is increasing. Then\\
			(a) If $X\le_{disp} Y$, then $J^{w_1}(\textbf{X}_{RSS}^{(n)})\le J^{w_1}(\textbf{Y}_{RSS}^{(n)})$.\\
			(b)  If $X\ge_{disp} Y$, then $J^{w_1}(\textbf{X}_{RSS}^{(n)})\ge J^{w_1}(\textbf{Y}_{RSS}^{(n)})$.
		\end{corollary}
		
		\begin{remark}
			For $w_1(x)=x$, the result in Corollary  \ref{cor com rss1} was proved by Qiu and Raqab (2022).
		\end{remark}

		\begin{lemma}\label{lemma1} [Ahmed et al. (1986); also see Qiu and Raqab (2022), lemma 4.3]
			Let $X$ and $Y$  be nonnegative random variables with pdf's f and g , respectively, satisfying $f(0)\ge g(0)>0$. If $X\le_{su}Y$ (or $X\le_{*}Y$ or $X\le_{c}Y$), then $X\le_{disp}Y$.
		\end{lemma}
		
		One may refer Shaked and Shantikumar (2007) for details of convex transform order ($\leq_c$), star order ($\leq_{\star}$), super additive order ($\leq_{su}$), and dispersive order ($\leq_{disp}$). In view of Theorem \ref{thm com rss1}  and Lemma \ref{lemma1}, the following result is obtained.
		\begin{theorem}
			Let $X$  and $Y$ be nonnegative random variables with pdf's f and g, cdf's F and G, respectively having $u_X=u_Y<\infty$.\\
			(a)  If $w_1$ is increasing, $w_1(x)\geq w_2(x)$ and $X\le_{su} Y$ (or $X\le_{*}Y$ or $X\le_{c}Y$), then $J^{w_1}(\textbf{X}_{RSS}^{(n)})\le J^{w_2}(\textbf{Y}_{RSS}^{(n)})$.\\
			(b)   If $w_1$ is increasing, $w_1(x)\leq w_2(x)$ and $X\ge_{su} Y$  (or $X\ge_{*}Y$ or $X\ge_{c}Y$), then $J^{w_1}(\textbf{X}_{RSS}^{(n)})\ge J^{w_2}(\textbf{Y}_{RSS}^{(n)})$.
		\end{theorem}
		
		If we take $w_1(x)=w_2(x)$ in above theorem, then we have following corollary.
		\begin{corollary}\label{cor com rss2}
			Let $X$  and $Y$ be nonnegative random variables with pdf's f and g, cdf's F and G, respectively having $u_X=u_Y<\infty$ and  $w_1$ is increasing.\\
			(a) If $X\le_{su} Y$ (or $X\le_{*}Y$ or $X\le_{c}Y$), then $J^{w_1}(\textbf{X}_{RSS}^{(n)})\le J^{w_1}(\textbf{Y}_{RSS}^{(n)})$.\\
			(b)  If $X\ge_{su} Y$  (or $X\ge_{*}Y$ or $X\ge_{c}Y$), then $J^{w_1}(\textbf{X}_{RSS}^{(n)})\ge J^{w_1}(\textbf{Y}_{RSS}^{(n)})$
		\end{corollary}
		
		\begin{remark}
			For $w_1(x)=x$, the result in Corollary  \ref{cor com rss2} was proved by Qiu and Raqab (2022).
		\end{remark}
		
		One may ask that whether the condition $X\le_{disp}Y$ in Theorem \ref{thm com rss1}  may be relaxed by $J^{w_1}(X)\leq J^{w_2}(Y)$. The following result give positive answer to this assertion.
		
		\begin{theorem}\label{thm com rss3}
			Let $X$  and $Y$ be nonnegative random variables with pdf's f and g, cdf's F and G. Let $\Delta (u)=w_1(F^{-1}(u))f(F^{-1}(u))-w_2(G^{-1}(u))g(G^{-1}(u))$,
			\[A_1=\{0\le u \le 1|\Delta (u)>0\},\ A_2=\{0\le u \le 1|\Delta (u)<0\}.\]
			If $ \inf_{A_1}\phi_{2i-1:2n-2i}(u)\geq  \sup_{A_2}\phi_{2i-1:2n-2i}(u)$, and if $J^{w_1}(X)\leq J^{w_2}(Y)$, then  $J^{w_1}(\textbf{X}_{RSS}^{(n)})$ $\le J^{w_2}(\textbf{Y}_{RSS}^{(n)})$
		\end{theorem}
		\begin{proof}
			Since $J^{w_1}(X)\leq J^{w_2}(Y)$, we have
			\begin{equation}\label{delta}
				\int_{0}^{1}\Delta (u)du=\int_{0}^{1}\left[w_1(F^{-1}(u))f(F^{-1}(u))-w_2(G^{-1}(u))g(G^{-1}(u))\right]du\geq 0.
			\end{equation}
			Now for $1\le i\le n,$ we have
			\begin{align*}
				&\int_{0}^{1}w_1(F^{-1}(u))f(F^{-1}(u))\phi_{2i-1:2n-2i}(u)du \\
&\ \ \ \ \ \ \ \ \ \ \ \ \ -\int_{0}^{1}w_2(G^{-1}(u))g(G^{-1}(u))\phi_{2i-1:2n-2i}(u)du\\
				&=\int_{0}^{1}\Delta (u)\phi_{2i-1:2n-2i}(u)du\\
				&=\int_{A_1}\Delta (u)\phi_{2i-1:2n-2i}(u)du+\int_{A_2}\Delta (u)\phi_{2i-1:2n-2i}(u)du\\
				&\geq  \inf_{A_1}\phi_{2i-1:2n-2i}(u)\int_{A_1}\Delta (u)du+\sup_{A_2}\phi_{2i-1:2n-2i}(u)\int_{A_2}\Delta (u)du\\
				&\geq \sup_{A_2}\phi_{2i-1:2n-2i}(u)\int_{A_1}\Delta (u)du+\sup_{A_2}\phi_{2i-1:2n-2i}(u)\int_{A_2}\Delta (u)du\\
				&\geq \sup_{A_2}\phi_{2i-1:2n-2i}(u)\int_{0}^{1}\Delta (u)du\\
				&\geq 0,
			\end{align*}
			where the inequalities follows by using the assumption that  \[ \inf_{A_1}\phi_{2i-1:2n-2i}(u)\geq  \sup_{A_2}\phi_{2i-1:2n-2i}(u)\] and the equation (\ref{delta}). Now using 
			equation (\ref{RSS1}) and above ineqaulity result follows.\hfill $\blacksquare$ 
		\end{proof}
		
		If we take $w_1(x)=w_2(x)$ in above theorem, then we have following corollary.
		\begin{corollary}\label{cor com rss3}
			Let $X$  and $Y$ be nonnegative random variables with pdf's f and g, cdf's F and G. Let $\Delta (u)=w_1(F^{-1}(u))f(F^{-1}(u))-w_1(G^{-1}(u))g(G^{-1}(u))$,
			\[A_1=\{0\le u \le 1|\Delta (u)>0\},\ A_2=\{0\le u \le 1|\Delta (u)<0\}.\]
			If  $\ \inf_{A_1}\phi_{2i-1:2n-2i}(u)\geq  \sup_{A_2}\phi_{2i-1:2n-2i}(u)$, and if $J^{w_1}(X)\leq J^{w_1}(Y)$, then  $J^{w_1}(\textbf{X}_{RSS}^{(n)})$ $\le J^{w_1}(\textbf{Y}_{RSS}^{(n)})$
		\end{corollary}
		
		\begin{remark}
			For $w_1(x)=x$, the result in Corollary  \ref{cor com rss3} was proved by Qiu and Raqab (2022).
		\end{remark}

		\section{Monotone properties}
		Here, we  provide a propostion having the monotone properties between Ranked Set Sampling's elements. The proof of the proposition is similar to Proposition 5.1 given in Qiu and Raqab (2022), hence omitted.
		\begin{proposition}
			Let $X$ be a random variable with pdf f and cdf F. If $\Lambda_X^{w_1}$ is decreasing, then
			\[J^{w_1}(X_{(1:n)1})\leq J^{w_1}(X_{(2:n)2})\le \ldots\le J^{w_1}(X_{(m:n)m});\]
			and if $\Lambda_X^w$ is increasing, then
			\[J^{w_1}(X_{(m:n)m})\leq J^{w_1}(X_{(m+1:n)m+1})\le \ldots\le J^{w_1}(X_{(n:n)n}),\]
			where $m=\left[\frac{n}{2}\right]$ the least integer greater than or equal to $\frac{n}{2}$ and $\left[ . \right]$ is the ceiling function such that $[x]$ maps $x$ to the least integer greater than or equal to $x$.
		\end{proposition}
		If we take $w_1(x)=x$ in above  proposition, we obtain  Proposition 5.1 of  Qiu and Raqab (2022)\\
		\\
		\textbf{ \Large Funding} \\
		\\
		Santosh Kumar Chaudhary would like to acknowledge financial support from the Council of Scientific and Industrial Research (CSIR) ( File Number 09/0081(14002) /2022-
		EMR-I ), Government of India. \\
		\\
		\textbf{ \Large Conflict of interest} \\
		\\
		No conflicts of interest are disclosed by the authors.

		\vspace{.3in}
		
		\noindent
		{\bf  Nitin Gupta} \\
		Department of Mathematics,\\
		Indian Institute of Technology Kharagpur\\
		Kharagpur-721302, INDIA\\
		E-mail: nitin.gupta@maths.iitkgp.ac.in
		
		\vspace{.1in}
		\noindent
		{\bf Santosh Kumar Chaudhary}\\
		Department of Mathematics,\\
		Indian Institute of Technology Kharagpur\\
		Kharagpur-721302, INDIA\\
		E-mail: skchaudhary1994@kgpian.iitkgp.ac.in\\

	\end{document}